\documentclass[a4paper, reqno]{amsart}
\usepackage[T1]{fontenc}
\usepackage[english]{babel}
\usepackage{amsmath}
\usepackage{amsthm}
\usepackage{xfrac}
\usepackage{amssymb, enumerate}
\usepackage{amsfonts}
\usepackage{dsfont}
\usepackage{graphicx}
\usepackage[lofdepth,lotdepth]{subfig}
\usepackage{mathtools,mathrsfs} 
\usepackage{cases} 
\usepackage{tikz}

\usepackage[shortlabels]{enumitem}

\usepackage{hyperref} 
\usepackage[capitalise, nameinlink, noabbrev]{cleveref} 
\hypersetup{colorlinks=true,
	linkcolor=blue,
	citecolor=red,
}

\usepackage{amsbsy,upref,color,amscd}
\usepackage[active]{srcltx}
\usepackage{bbm}
\usepackage{algorithm}
\usepackage{algorithmicx}
\usepackage{algpseudocode}
\usetikzlibrary{shapes,snakes,patterns}
\usetikzlibrary{decorations.text}
\usetikzlibrary{decorations.pathmorphing, decorations.text}
\usepackage{subfig}
\usepackage{eurosym}
\usepackage{fancyhdr}
\usepackage{colorspace}

\usepackage{environ}
\makeatletter
\newsavebox{\measure@tikzpicture}
\NewEnviron{scaletikzpicturetowidth}[1]{%
  \def\tikz@width{#1}%
  \begin{lrbox}{\measure@tikzpicture}%
  \BODY
  \end{lrbox}%
  \pgfmathparse{#1/\wd\measure@tikzpicture}%
  \BODY
}
\makeatother
\usepackage{sansmath}
\usetikzlibrary{shadings,intersections,patterns,calc,angles}

\usepackage{rotating,multirow,booktabs,float}
\usepackage{pgfplots}

\pgfplotsset{compat=1.10}
\usepgfplotslibrary{fillbetween}
\usepgfplotslibrary{polar}
\usepackage{tkz-euclide}
\usepackage{environ}

\usetikzlibrary{arrows,automata}
\usetikzlibrary{positioning}
\usetikzlibrary{calc,through}

\tikzset{
    state/.style={
           rectangle,
           rounded corners,
           draw=black, very thick,
           minimum height=2em,
           inner sep=2pt,
           text centered,
           },
}

\usepackage{pgf}
\usepackage{xcolor}

\usepackage{etoolbox}%
\usepackage{cfr-lm}

\def\ds{\displaystyle}
\def\eps{{\varepsilon}}
\def\N{\mathbb{N}}

\def\R{\mathbb{R}}

\def\HH{\mathcal{H}}

\definecolor{racing}{rgb}{0.7,0.1,0.2}
\definecolor{french}{rgb}{0,0.2,0.7}

\newcommand{\diff}{\!\setminus\!}

\newcommand{\Om}{\Omega}

\def\HH{\mathcal{H}}

\def\XXint#1#2#3{{\setbox0=\hbox{$#1{#2#3}{\int}$ }
\vcenter{\hbox{$#2#3$ }}\kern-.6\wd0}}

\newtheorem{proposition}{Proposition}[section]
\newtheorem{theorem}[proposition]{Theorem}
\newtheorem{corollary}[proposition]{Corollary}
\newtheorem{lemma}[proposition]{Lemma}

\theoremstyle{definition}
\newtheorem{definition}[proposition]{Definition}
\newtheorem{remark}[proposition]{Remark}

\title[Epsilon-regularity for a free boundary system]{Epsilon-regularity for the solutions of a free boundary system}

\author{Francesco Paolo Maiale, Giorgio Tortone and Bozhidar Velichkov }

\address {Francesco Paolo Maiale \newline \indent
	Scuola Normale Superiore\newline \indent
	Piazza dei Cavalieri 7, 56126 Pisa, Italy}
\email{francesco.maiale@sns.it}

\address {Giorgio Tortone \newline \indent
	Dipartimento di Matematica, Universit\`a di Pisa \newline \indent
	Largo Bruno Pontecorvo, 5, I--56127 Pisa, Italy}
\email{giorgio.tortone@dm.unipi.it}

\address {Bozhidar Velichkov \newline \indent
Dipartimento di Matematica, Universit\`a di Pisa \newline \indent
Largo Bruno Pontecorvo, 5, I--56127 Pisa, Italy}
\email{bozhidar.velichkov@unipi.it}

\thanks{{\bf Acknowledgments.}
G.T. and B.V. are supported by the European Research Council (ERC), EU Horizon 2020 programme, through the project ERC VAREG - \it Variational approach to the regularity of the free boundaries \rm (No. 853404).}


\begin{document}

\begin{abstract}

This paper is dedicated to a free boundary system arising in the study of a class of shape optimization problems. The problem involves three variables: two functions $u$ and $v$, and a domain $\Omega$; with $u$ and $v$ being both positive in $\Omega$, vanishing simultaneously on $\partial\Omega$ and satisfying an overdetermined boundary value problem involving the product of their normal derivatives on $\partial\Omega$. Precisely, we consider solutions $u, v \in C(B_1)$ of 
$$-\Delta u= f \quad\text{and} \quad-\Delta v=g\quad\text{in}\quad \Omega=\{u>0\}=\{v>0\}\ ,\qquad \frac{\partial u}{\partial n}\frac{\partial v}{\partial n}=Q\quad\text{on}\quad \partial\Omega\cap B_1.$$
Our main result is an epsilon-regularity theorem for viscosity solutions of this free boundary system. We prove a partial Harnack inequality near flat points for the couple of auxiliary functions $\sqrt{uv}$ and $\frac12(u+v)$. Then, we use the gained space near the free boundary to transfer the improved flatness to the original solutions. Finally, using the partial Harnack inequality, we obtain an improvement-of-flatness result, which allows to conclude that flatness implies $C^{1,\alpha}$ regularity.
\end{abstract}
\maketitle

\section{Introduction}\label{s:intro}

Let $u, v \in C(B_1)$ be two continuos non-negative functions on the unit ball in $\R^d$ such that
$$\Omega:=\{u>0\}=\{v>0\}.$$
Suppose that $u$ and $v$ are also solutions of the free boundary problem
\begin{align}
-\Delta u=0&\quad\text{in}\quad \Omega,\label{e:equation-u}\\
-\Delta v=0&\quad\text{in}\quad\Omega,\label{e:equation-v}\\
\frac{\partial u}{\partial n}\frac{\partial v}{\partial n}=1&\quad\text{on}\quad \partial\Omega\cap B_1,\label{e:boundary_condition}
\end{align}
where the two equations \eqref{e:equation-u} and \eqref{e:equation-v} hold in the classical sense in the open set $\Omega$. On the other hand, since we will not assume that $\Omega$ is regular, the boundary condition \eqref{e:boundary_condition} is to be intended in a generalized sense. Following the classical approach of Caffarelli \cite{caf1,caf2}, for simplicity, in the introduction and in \cref{t:main}, we will assume that  \eqref{e:boundary_condition} holds in the sense of \cref{d:solutions} below. Our main $\eps$-regularity result applies to an even more general notion of solution, but in order to avoid technicalities in the introduction, we postpone this discussion to \cref{s:the-notion-of-solution}.

\begin{definition}[Definition of solutions]\label{d:solutions}
We say that \eqref{e:boundary_condition} holds,
if at any point $x_0\in \partial\Omega\cap B_1$ at which $\partial\Omega$ admits a one-sided tangent ball we have that the functions $u$ and $v$ can be expanded as
\begin{equation}
u(x)=\alpha \big((x-x_0)\cdot\nu\big)_++\ o\big(|x-x_0|\big),
\end{equation}
\begin{equation}
v(x)=\beta \big((x-x_0)\cdot \nu\big)_++\ o\big(|x-x_0|\big),
\end{equation}
where $\nu$ is a unit vector and $\alpha$ and $\beta$ are positive real numbers such that $\alpha\beta=1.$
\end{definition}

Our main result is a regularity theorem which applies to solutions, which are sufficiently flat in the sense of the following defintion.

\begin{definition}[Definition of flatness]\label{d:flatness}
We say that $u$ and $v$ are $\varepsilon$-flat in $B_1$, if there is a unit vector $\nu\in \partial B_1$ and  positive constants  $\alpha$ and $\beta$ such that $\alpha\beta=1$ and
\begin{equation}
\alpha\,\big(x\cdot\nu-\varepsilon\big)_+\le u(x)\le \alpha\,\big(x\cdot\nu+\varepsilon\big)_+\quad\text{for every}\quad x\in B_1,
\end{equation}	
\begin{equation}
\beta\,\big(x\cdot\nu-\varepsilon\big)_+\le v(x)\le \beta\,\big(x\cdot\nu+\varepsilon\big)_+\quad\text{for every}\quad x\in B_1.
\end{equation}	
We will also say that $u$ and $v$ are $\varepsilon$-flat in the direction $\nu$.
\end{definition}

We prove the following theorem.

\begin{theorem}\label{t:main}	
There is a constant $\varepsilon_0>0$ such that the following holds.\\ Let $u$ and $v$ be two non-negative continuous functions on $B_1$ and let $\Omega:=\{u>0\}=\{v>0\}.$\\
If $u$ and $v$ are solutions of \eqref{e:equation-u}-\eqref{e:equation-v}-\eqref{e:boundary_condition} and are $\varepsilon$-flat in $B_1$, for some $\varepsilon\in(0,\varepsilon_0]$, then $\partial\Omega$ is $C^{1,\alpha}$ in $B_{\sfrac12}$.
\end{theorem}

\cref{t:main} follows directly from \cref{t:main-main} (\cref{s:main-main}), in which we prove the same result for a more general notion of solution, which we define in \cref{s:the-notion-of-solution} in terms of the blow-ups of $u$ and $v$. The proof of \cref{t:main-main} will be given in \cref{s:partial_harnack} and \cref{s:improvement_of_flatness}.\\ 

\smallskip

The rest of the introduction is organized as follows. In \cref{sub:free-boundary} we briefly discuss the relation of the system \eqref{e:equation-u}-\eqref{e:equation-v}-\eqref{e:boundary_condition} to the well-known one-phase, two-phase and vectorial Bernoulli problems, with which it shares several key features. In \cref{sub-sketch-of-the-proof} we briefly explain the overall strategy and the novelties of the proof. Finally, in \cref{sub:shape-optimization}, we discuss the applications of our result to the theory of Shape Optimization.

\subsection{The classical one-phase, two-phase and vectorial problems}\label{sub:free-boundary} The free boundary problem \eqref{e:equation-u}-\eqref{e:equation-v}-\eqref{e:boundary_condition} is a vectorial analogue of the following classical one-phase problem
\begin{equation}\label{e:one-phase}
\begin{cases}
\Delta u=0&\text{in}\quad \Omega:=\{u>0\},\\
|\nabla u|=1&\text{on}\quad \partial\Omega\cap B_1,
\end{cases}
\end{equation}
which was introduced by Alt and Caffarelli in \cite{alca} in the early 80s.
Later, in a series of papers (see \cite{caf1,caf2} and the book \cite{cs}), Caffarelli studied the following two-phase problem, in which the solution is given by a single function $u:B_1\to\R$ that changes sign:
\begin{equation}\label{e:two-phase}
\begin{cases}
\Delta u=0&\text{in}\quad \Omega_+:=\{u>0\}\quad\text{and}\quad \Omega_-:=\{u<0\}\ ,\\
(\partial_n^+ u)^2-(\partial_n^- u)^2=1&\text{on}\quad \partial\Omega_+\cap\partial\Omega_-\cap B_1 ,\\
 \end{cases}
\end{equation}
 and where the transmission condition on the boundary $\partial\Omega_+\cap\partial\Omega_-$ is defined in terms of Taylor expansions at points with one-sided tangent ball (contained in $\Omega_+$ or $\Omega_-$), exactly as in \cref{d:solutions}; here  $n$ denotes the normal to $\partial\Omega_+\cap\partial\Omega_-$ at such points.
More recently, De Silva \cite{desilva} gave a different proof to the one-phase $\eps$-regularity theorem from \cite{alca}; the method found application to several generalizations of \eqref{e:two-phase} (see \cite{dfs,dfs2,dfs3}) and opened the way to the original two-phase problem of Alt-Caffarelli-Friedman \cite{acf}, for which the $C^{1,\alpha}$ regularity of the free boundary in every dimension was proved only recently in \cite{despve} by a similar argument. \medskip

 Inspired by a problem arising in the theory of shape optimization, a cooperative vectorial version of the one-phase problem was introduced in \cite{csy}, \cite{krli} and \cite{mtv}. In this case, the solutions are vector-valued functions  $$U=(u_1,\dots,u_k):B_1\to\R^k$$ satisfying
  \begin{equation}\label{e:vectorial}
 \Delta U=0\quad\text{in}\quad \Omega:=\{|U|>0\}\ ,\qquad
 \sum_{j=1}^k|\nabla u_j|^2=1\quad\text{on}\quad \partial\Omega\cap B_1.
 \end{equation}
The regularity of the vectorial free boundaries turned out to be quite challenging, especially when it comes to viscosity solutions. This is mainly due to the fact that the regularity techniques, based on the maximum principle and on the comparison of the solutions with suitable test functions (see for instance \cite{desilva} and \cite{despve}), are in general hard to implement in the case of systems. Epsilon-regularity theorems for the vectorial problem \eqref{e:vectorial} were proved in \cite{csy, krli, krli2, mtv, mtv2, spve} and more recently, in \cite{desilva-tortone}, where the regularity of the flat free boundaries was obtained directly for viscosity solutions.

\medskip

\subsection{Outline of the paper and sketch of the proof}\label{sub-sketch-of-the-proof} The free boundary problem \eqref{e:equation-u}-\eqref{e:equation-v}-\eqref{e:boundary_condition} is also a vectorial problem and arises in the study of a whole class of shape optimization problems (which we will discuss in \cref{sub:shape-optimization}). On the other hand, unlike the one-phase (\cite{alca}), the two-phase (\cite{acf}) and the vectorial problem \cite{csy,mtv} it doesn't have an underlying variational structure in terms of the functions $u$ and $v$, thus purely variational arguments as the epiperimetric inequality (see \cite{spve}) cannot be applied.

Thus, we prove an improvement-of-flatness result for solutions of \eqref{e:equation-u}-\eqref{e:equation-v}-\eqref{e:boundary_condition} from which \cref{t:main} follows by a standard argument (for more details we refer for instance to \cite{velectures}). Precisely, the aim of the paper is to prove the following theorem (see also the more general \cref{t:improvement_of_flatness}).

\begin{theorem}[Improvement of flatness]\label{t:improvement_of_flatness-intro}
There are dimensional constants $\eps_0>0$ and $C>0$ such that the following holds.
 Let $u:B_1\to\R$ and $v:B_1\to\R$ be two continuous non-negative functions, which are also solutions to \eqref{e:equation-u}-\eqref{e:equation-v}-\eqref{e:boundary_condition}.
Let $\Omega:=\{u>0\}=\{v>0\}$ and $0\in\partial\Omega$.
	If $u$ and $v$ satisfy
	$$(x_d - \eps)_+ \le u(x) \le (x_d + \eps)_+ \qquad \text{and} \qquad  (x_d  - \eps)_+ \le v(x) \le (x_d + \eps)_+ \qquad\text{in } B_1,$$
	for some $\eps < \eps_0$, then there are a unit vector $\nu \in \R^d$ with $|\nu - e_d| \le C \eps$ and a radius $\rho\in(0,1)$ such that
	$$ \tilde{\alpha}(x \cdot \nu - \sfrac{\eps}2)_+ \le \frac{u(\rho x)}{\rho} \le \tilde{\alpha}(x \cdot \nu + \sfrac{\eps}2)_+ \quad \text{and} \quad\tilde{\beta}(x \cdot \nu - \sfrac{\eps}2)_+ \le \frac{v(\rho x)}{\rho}  \le \tilde{\beta}(x \cdot \nu + \sfrac{\eps}2)_+  $$
	for all $x \in B_1$, where $\tilde{\alpha}$ and $\tilde{\beta}$ are such that $\widetilde\alpha\widetilde\beta=1$, $|1 - \tilde{\alpha}|\le C\eps$ {and} $|1-\tilde{\beta}| \le C\eps.$
\end{theorem}

In order to prove \cref{t:improvement_of_flatness-intro}, we use the general strategy of De Silva developed in \cite{desilva} for viscosity solutions of the one-phase problem, which reduces the proof of \cref{t:improvement_of_flatness-intro} to two key ingredients (partial Harnack inequality and analysis of the linearized problem) at which are concentrated the whole difficulty of the proof and the insight on the specific problem. The idea is the following. Arguing by contradiction, one considers a sequence of solutions $(u_n,v_n)$, which are $\eps_n$-flat with $\eps_n\to 0$, and then produces the linearizing sequence
$$\widetilde u_n(x)=\frac{u_n(x)-x_d}{\eps_n}\qquad\text{and}\qquad \widetilde v_n(x)=\frac{v_n(x)-x_d}{\eps_n}.$$
The argument now can be divided in two main steps.

The first step is to show that $\widetilde u_n$ and $\widetilde v_n$ converge (see \cref{s:partial_harnack}) to some functions $u_\infty$ and $v_\infty$; this is done  by proving a {\it partial Harnack inequality}, which is the hardest part of the proof. Roughly speaking, the partial Harnack inequality in our case (see \cref{t:partial_harnack1} for the precise statement) states that if $u$ and $v$ is any couple of $\eps$-flat solutions, for some $\eps<\eps_0$, then there is a constant $c\in(0,1)$ such that
\begin{equation}\label{e:ph1-intro}
(x_d - (1-c)\eps)_+ \le u(x) \le (x_d + \eps)_+ \qquad \text{and} \qquad  (x_d  - (1-c)\eps)_+ \le v(x) \le (x_d + \eps)_+ \quad\text{in}\quad B_{\sfrac12},
\end{equation}
or
\begin{equation}\label{e:ph2-intro}
(x_d - \eps)_+ \le u(x) \le (x_d + (1-c)\eps)_+ \qquad \text{and} \qquad  (x_d  - \eps)_+ \le v(x) \le (x_d + (1-c)\eps)_+ \quad\text{in}\quad B_{\sfrac12},
\end{equation}
that is, the flatness is improved from above or from below, in the same direction $e_d$, but without the scaling factor that would allow to iterate the statement without going above the threshold $\eps_0$.

 The second step is to show that the $u_\infty$ and $v_\infty$ are solutions to a PDE problem (the so-called {\it linearized problem} or  {\it limit problem}), from which one can obtain an oscillation decay for $u_\infty$ and $v_\infty$ that can then be transfered back to $\widetilde u_n$ and $\widetilde v_n$, for some $n$ large enough. In our case, the linearized problem is the following system of PDEs on the half-ball $B_1\cap\{x_d>0\}$ (see \cref{l:linearized}, \cref{s:improvement_of_flatness}):
\begin{equation}\label{e:linearized-pb-intro}
\begin{cases}
\Delta u_\infty=\Delta v_\infty = 0&\text{in}\quad B_1\cap\{x_d>0\},\\
u_\infty=v_\infty\quad\text{and}\quad \partial_{x_d} u_\infty+\partial_{x_d} v_\infty=0&\text{on}\quad B_1\cap\{x_d=0\}.\end{cases}
\end{equation}

In our case, the most challenging part of the proof is the partial Harnack inequality. In fact, in the one-phase and the two-phase problems (see \cite{desilva}, \cite{dfs,dfs2,dfs3}, and \cite{despve}) the validity of \eqref{e:ph1-intro}-\eqref{e:ph2-intro} is obtained by constructing explicit competitors, which are essentially variations of the constructions in \cite{desilva} . In our case though, the functions $u$ and $v$, considered separately, are not solutions (not even sub- or supersolutions) to any free boundary problem. Thus, a key observation in our case, which is insipred by \cite{desilva-tortone} and turns out to be crucial in both the partial Harnack inequality and the proof of \eqref{e:linearized-pb-intro}, is that if $u$ and $v$ are solutions to  \eqref{e:equation-u}-\eqref{e:equation-v}-\eqref{e:boundary_condition}, then $\sqrt{uv}$ and $\frac12(u+v)$ are respectively a viscosity subsolution and viscosity supersolution of the one-phase problem \eqref{e:one-phase} (see \cref{l:viscosity} and \cref{r:superharmonicity})\footnote{This situation is similar to the one of the vectorial problem \eqref{e:vectorial} in which each of the components of $U$ is a viscosity supersolution, while, as it was shown in \cite{mtv}, the modulus $|U|$ satisfies $|\nabla|U||=1$ on $\partial\{|U|>0\}$ and is a viscosity subsolution of \eqref{e:one-phase}; this information was used in \cite{desilva-tortone} to prove a partial Harnack inequality and an $\eps$-regularity theorem for the solutions of \eqref{e:vectorial}. }. Moreover, it is easy to check that both $\sqrt{uv}$ and $\frac12(u+v)$ hereditate the flatness of $u$ and $v$. Thus, by using the competitors from \cite{desilva} on these functions, we obtain the following dichotomy in $B_{\sfrac12}$ :
\begin{center}
\it the flatness $\frac12(u+v)$ is improved from above \rm  \quad or \quad \it the flatness of $\sqrt{uv}$ is improved from below.\rm
\end{center}
Now, notice that we cannot transfer this information back to $u$ and $v$ just be an algebric manipulation; for instance, a bound from below on $\sqrt{uv}$ does not a priori imply a bound from below on both $u$ and $v$. On the other hand, one can easily notice that the improved flatness of $\sqrt{uv}$ or $\frac12(u+v)$, in particular, implies that in $B_{\sfrac12}$ the boundary $\partial\Omega$ is trapped between two nearby translations of a half-space, which are distant at most $(2-c)\eps$. Using this geometric information and a comparison argument based on the boundary Harnack principle, in \cref{t:gain_space}, we obtain that also the flatness of $u$ and $v$ improves in $B_{\sfrac12}$.

\subsection{On the boundary condition $\frac{\partial u}{\partial n}\frac{\partial v}{\partial n}=1$ and its relation to a shape optimization problem}\label{sub:shape-optimization}
Our result applies to a whole class of shape optimization problems, that is, variational problems of the form
$$\min\Big\{J(\Omega)\ :\ \Omega\in\mathcal A\Big\},$$
where $\mathcal A$ is an admissible class of subsets of $\R^d$ and $J$ is a given function on $\mathcal A$. Typically, the admissible set $\mathcal A$ is a family of sets of fixed measure, contained  in a given bounded open set $D\subset\R^d$, while the functional $J$ is monotone with respect to the set inclusion and depends on the resolvent of an elliptic operator with Dirichlet boundary conditions on $\partial\Omega$.
The shape functionals are usually related to models in Engineering, Mechanics and Material Sciences and in most of the cases fall in one of the following main classes: \it  spectral functionals \rm and \it integral functionals \rm (for more details we refer to the books \cite{bubu}, \cite{henrot}, \cite{henrotbook}, \cite{hepi} and the survey paper \cite{bu}).

The spectral functionals are functionals of the form $J(\Omega)=F(\lambda_1(\Omega),\dots,\lambda_k(\Omega))$, where $F:\R^k\to\R$ is a monotone (in each variable) function and $\lambda_1(\Omega),\dots,\lambda_k(\Omega)$ are the eigenvalues of the Dirichlet Laplacian on $\Omega$. The regularity and the local structure of these optimal sets were studied in \cite{bmpv}, \cite{krli, krli2}, \cite{mtv} (see also \cite{brla} and \cite{matrve} for the special cases $J(\Omega)=\lambda_1(\Omega)$ and $J(\Omega)=\lambda_2(\Omega)$) and are related to the vectorial Bernoulli problem from \cite{csy}, \cite{mtv} and  \cite{mtv2}. An $\eps$-regularity theorem for general spectral functionals was obtained in \cite{krli2}.

The integral functionals can be written in the general form
\begin{equation}\label{e:sop}
J(\Omega)=\int_{D}j(u_\Omega,x)\,dx\,,
\end{equation}
where $j:\R\times D\to\R$ is a given function and the state function $u_\Omega$ is the unique solution of
$$-\Delta u=f\quad\text{in}\quad\Omega\ ,\qquad u\in H^1_0(\Omega),$$
the right-hand side $f:D\to\R$ being a fixed measurable function.

As it was observed already in \cite{buve}, a free boundary system  of the form \eqref{e:equation-u}-\eqref{e:equation-v}-\eqref{e:boundary_condition} naturally arises in the computation of the first variation of $J$. This is easy to see if one computes formally the first variation of $J$ for smooth sets. Indeed, if we suppose that $\Omega$ is an optimal set and smooth, and that $\xi\in C^\infty(D;\R^d)$ is a compactly supported vector field, then we can define the family of sets $\Omega_t:=(Id+t\xi)(\Omega)$ and the family of state functions $u_t:=u_{\Omega_t}$. Then, the first variation of $J$ is given by
\begin{align*}
\delta J(\Omega)[\xi]:=\frac{d}{dt}\Big|_{t=0}J(\Omega_t)
&=\frac{d}{dt}\Big|_{t=0}\int_{\Omega_t}j(u_t,x)\,dx\\
&=\frac{d}{dt}\Big|_{t=0}\left[\int_{D}\big(j(u_t,x)-j(0,x)\big)\,dx+\int_{\Omega_t}j(0,x)\,dx\right]\\
&=\int_{\Omega}u'\frac{\partial j}{\partial u}(u_\Omega,x)\,dx+\int_{\partial \Omega} j(0,x)\,\xi\cdot n_\Omega\,d\HH^{d-1},
\end{align*}
where $n_\Omega(x)$ to be the exterior normal at $x\in\partial\Omega$ and the formal derivative $u'$ (of $u_t$ at $t=0$) is the solution of the boundary value problem
$$\Delta u'=0\quad\text{in}\quad \Omega\,,\qquad u'=-\xi\cdot\nabla u_\Omega\quad\text{on}\quad \partial\Omega,$$
in which the condition on $\partial\Omega$ is a consequence of the fact that, given $x\in \partial\Omega$, we have
$$u_t(x+t\xi(x))=0\quad\text{for every}\quad t\in\R.$$
We next define the function
$$g(x):=-\frac{\partial j}{\partial u}(u_\Omega(x),x),$$
and the solution $v_\Omega$ of the problem
$$-\Delta v=g\quad\text{in}\quad\Omega\ ,\qquad v\in H^1_0(\Omega).$$
\begin{remark}
Before we continue with the computation of the first variation, we notice that, in order to have the monotonicity of $J$, it is natural to assume that $f\ge 0$ and $\frac{\partial j}{\partial u}\le 0$, which of course implies that $g\ge 0$ and that both $u_\Omega$ and $v_\Omega$ are non-negative. On the other hand, if $f$ and $\frac{\partial j}{\partial u}$ change sign, then in general an optimal set might not exist (see for instance \cite{buve}).
\end{remark}	
	
In order to complete the computation of $\delta J(\Omega)[\xi]$, we integrate by parts in $\Omega$, obtaining
$$-\int_{\Omega}u'g(x)\,dx=\int_{\Omega}u'\Delta v_\Omega\,dx=-\int_{\Omega}\nabla u'\cdot \nabla v_\Omega\,dx+\int_{\partial\Omega}u'\frac{\partial u_\Omega}{\partial n}=\int_{\partial\Omega}u'\frac{\partial u_\Omega}{\partial n}\ .$$
But now since $\nabla u_\Omega$ is parallel to $n_\Omega$ at the boundary, we have that $u'=-\xi\cdot\nabla u_\Omega=-(\xi\cdot n_\Omega)(n_\Omega\cdot \nabla u_\Omega)$. Thus, the first variation of $J$ is
$$\delta J(\Omega)[\xi]=\int_{\partial\Omega}\Big(-\frac{\partial u_\Omega}{\partial n}\frac{\partial v_\Omega}{\partial n}+j(0,x)\Big)n_\Omega\cdot\xi\ .$$
Since the vector field $\xi$ is arbitrary and since $\Omega$ is a minimizer among the sets of prescribed measure, we get that in a neighborhood $B_r(x_0)$ of a point $x_0$ of the free boundary $\partial\Omega\cap D$, $u_\Omega$ and $v_\Omega$ are solutions of the system
$$
\begin{cases}
-\Delta u=f&\quad\text{in}\quad \Omega\cap B_r(x_0),\\
-\Delta v=g&\quad\text{in}\quad\Omega\cap B_r(x_0),\smallskip\\
\displaystyle{\frac{\partial u}{\partial n}\frac{\partial v}{\partial n}}=c+j(0,x)&\quad\text{on}\quad \partial\Omega \cap B_r(x_0),
\end{cases}
$$
where $c$ is a positive constant.

Our definition \cref{d:solutions} is a generalization of the notion of solution and was proposed by Caffarelli in \cite{caf1,caf2} in the context of a two-phase free boundary problem. In a subsequent paper we will use Theorem \ref{t:main}  to obtain a regularity result for optimal sets for functionals of the form \eqref{e:sop}. We notice that in Theorem \ref{t:main} we do not assume that the functions $u$ and $v$ are minimizers of a functional or solutions of a shape optimization problem of any kind, so this result is of independent interest and can be seen as a one-phase vectorial version of the classical results of Caffarelli \cite{caf1,caf2}.

\section{On the viscosity formulations of solution}\label{s:the-notion-of-solution}
In this section we briefly discuss the boundary condition \eqref{e:boundary_condition}. In particular, in \cref{d:solutions2} we give a more general notion of solution, which we will use throughout the paper. We start by recalling the following definition.
\begin{definition}[One sided tangent balls]\label{d:one-sided-tangent-ball}
Let $\Omega\subset\R^d$ be an open set and let $x_0\in\partial\Omega$. We say that $\Omega$ admits a one-sided tangent ball at $x_0$ if one of the following conditions hold:	
	\begin{enumerate}[(i)]
	\item\label{item:inside} there are $r>0$ and $y_0\in\Omega$ such that
	$\ B_r(y_0)\subset \Omega\quad\text{and}\quad \partial B_r(y_0)\cap \partial \Omega=\{x_0\},\smallskip$
	\item\label{item:outside}  there are $r>0$ and $y_0\in \R^d\setminus \overline\Omega$ such that
	$\ B_r(y_0)\subset  \R^d\setminus \overline\Omega\quad\text{and}\quad \partial B_r(y_0)\cap \partial \Omega=\{x_0\}$.
\end{enumerate}	
Moreover, we will use the notation
\begin{equation}\label{e:def-solutions-unit-vector}
\displaystyle\nu_{x_0,y_0}:= \frac{y_0-x_0}{|y_0-x_0|}\quad\text{in the case \ref{item:inside}},\quad\text{and}\quad  \displaystyle\nu_{x_0,y_0}:=-\frac{y_0-x_0}{|y_0-x_0|}\quad\text{in the case \ref{item:outside}}.
\end{equation}
We notice that when $\Omega$ is regular, the vector $\nu_{x_0,y_0}$ is the inner normal to $\partial\Omega$ at $x_0$, while for non-smooth domains it may depend on the ball $B_r(y_0)$.
\end{definition}

Let $Q:B_1\to\R$ be a $C^{0,\alpha}$-regular function (for some $\alpha>0$) and suppose that:
\begin{equation}
\text{There is}\quad C_Q\ge 1\quad\text{such that}\quad C_Q^{-1}\le Q(x)\le C_Q\quad\text{for all}\quad x\in B_1.
\end{equation}
Then, \cref{d:solutions} can be generalized as follows.
\begin{definition}[Definition of solutions I]\label{d:solutions11}
Given two continuous non-negative functions $u,v:B_1\to\R$ with the same support $\Omega=\{u>0\}=\{v>0\}$,	we say that
	$$	\frac{\partial u}{\partial n}\frac{\partial v}{\partial n}=Q\quad\text{on}\quad \partial\Omega\cap B_1,$$
	if at any point $x_0\in \partial\Omega\cap B_1$, for which $\partial\Omega$ admits a one-sided tangent ball at $x_0$,
	we have that the functions $u$ and $v$ can be expanded as
	\begin{equation}
	u(x)=\alpha \left((x-x_0)\cdot\nu\right)_++\ o\big(|x-x_0|\big),
	\end{equation}
	\begin{equation}
	v(x)=\beta \left((x-x_0)\cdot \nu\right)_++\ o\big(|x-x_0|\big),
	\end{equation}
	where $\nu$ is the unit vector given by \eqref{e:def-solutions-unit-vector} and $\alpha$ and $\beta$ are positive real numbers such that:
	$$\alpha\beta=Q(x_0).$$
\end{definition}
In particular, this definition implies that if $\partial\Omega$ admits a one-sided tangent ball at $x_0\in\partial\Omega$, then this tangent ball is unique, which of course excludes a priori domains with angles and cusps. Moreover, it implies that at such points the blow-ups of $u$ and $v$ are unique. Now, since in many situations this is not a priori known, we will work with couples which are solutions in the more general sense described in the next subsection.

\subsection{A more general notion of solution}\label{sub:solutions}
For every $x_0\in\partial\Omega\cap B_1$ and every $r>0$ small enough, we define
$$u_{r,x_0}(x)=\frac1{r}u(x_0+rx)\qquad\text{and}\qquad v_{r,x_0}(x)=\frac1{r}v(x_0+rx).$$
{Throughout the paper we will also adopt the notation $u_{r}:=u_{r,0}$ and $ v_r:=v_{r,0}$.} Therefore, \cref{d:solutions11} can be expressed in terms of the rescaling $u_{r,x_0}$ and $v_{r,x_0}$ in the following way.
\begin{remark}\label{l:equivalent_def_solutions}
	Let $u,v:B_1\to\R$ be two non-negative continuous functions with the same support $$\Omega:=\{u>0\}=\{v>0\}.$$
	Then, the following are equivalent
	\begin{enumerate}[\rm (1)]
		\item $\ds\frac{\partial u}{\partial n}\frac{\partial v}{\partial n}=Q\,\text{ on }\, \partial\Omega\cap B_1,$
		in the sense of \cref{d:solutions11}.
		\item At any point $x_0\in \partial\Omega\cap B_1$, for which one of the conditions \ref{item:inside} and \ref{item:outside} of \cref{d:one-sided-tangent-ball} hold, we have that $u_{r,x_0}$ and $v_{r,x_0}$ converge uniformy in $B_1$ as $r\to 0$ respectively to the functions
		\begin{equation}
		x\mapsto \alpha \left((x-x_0)\cdot\nu\right)_+\qquad\text{and}\qquad x\mapsto \beta \left((x-x_0)\cdot\nu\right)_+
		\end{equation}
		where $\alpha$ and $\beta$ are positive constants such that $\alpha\beta=Q(x_0)$ and $\nu\in\R^d$ is the unit vector given by \eqref{e:def-solutions-unit-vector}.
	\end{enumerate}
\end{remark}	

In particular, \cref{l:equivalent_def_solutions} implies that \cref{d:solutions11} can be generalized as follows.

\begin{definition}[Definition of solutions II]\label{d:solutions2}
	Given two continuous non-negative functions $u,v:B_1\to\R$ with the same support $\Omega=\{u>0\}=\{v>0\}$,	we say that
	$$	\frac{\partial u}{\partial n}\frac{\partial v}{\partial n}=Q\quad\text{on}\quad \partial\Omega\cap B_1,$$
	if at any point $x_0\in \partial\Omega\cap B_1$, for which one of the conditions \ref{item:inside} and \ref{item:outside} of \cref{d:one-sided-tangent-ball} hold, there exist:
	\begin{itemize}
		\item a decreasing sequence $r_n\to 0$;
		\item two positive constants $\alpha>0$ and $\beta>0$ such that $\alpha\beta=Q(x_0)$;
		\item a unit vector $\nu\in\R^d$;
	\end{itemize}	
	such that the sequences $u_{r_n,x_0}$ and $v_{r_n,x_0}$ converge uniformly in $B_1$ respectively to the functions
	{\begin{equation}
u_0(x):=\alpha \left(x\cdot\nu\right)_+\qquad\text{and}\qquad v_0(x):=\beta \left(x\cdot\nu\right)_+\ .
	\end{equation}}
\end{definition}

\begin{remark}
We will say that $u_0$ and $v_0$ are blow-up limits of $u$ and $v$ at $x_0$. We notice that the blow-up limits at a point may not be unique as they a priori depend on the sequence $r_n\to0$.
\end{remark}

\begin{remark}
	\cref{t:main} holds also for solutions $u$ and $v$ of \eqref{e:equation-u}-\eqref{e:equation-v}-\eqref{e:boundary_condition} in the sense of \cref{d:solutions2}. In fact, the entire proof of this theorem will be given for solutions in the sense of \cref{d:solutions2}.	
\end{remark}	

\begin{remark}
The sequence $r_n\to0$ from \cref{d:solutions2} may depend on the tangent ball $B_r(y_0)$ at $x_0$ (which in turn may not be unique). Thus, in \cref{d:solutions2} we do not assume that the blow-ups of $u$ and $v$  at $x_0$, as well as the the tangent ball $B_r(y_0)$, are unique.  	
\end{remark}

\subsection{Optimality conditions in viscosity sense}\label{sub:viscosity}
{In view of the last notion of solutions, we can finally state the viscosity formulation of the free boundary condition introduced in \cref{d:solutions2}. }
\begin{definition}
Let $u:B_1\to\R$ be a continuous non-negative function, $\varphi\in C^\infty(\R^d)$ be given and
$$\varphi_+(x):=\max\{\varphi(x),0\}.$$
\begin{itemize}
\item We say that $\varphi_+$ touches $u$ from 	below at a point $x_0\in\partial\{u>0\}\cap B_1$ if $u(x_0)=\varphi(x_0)=0$ and
$$\varphi_+(x)\le u(x)\quad\text{for every $x$ in a neighborhood of $x_0$}.$$
\item We say that $\varphi_+$ touches $u$ from 	above at a point $x_0\in\partial\{u>0\}\cap B_1$ if $u(x_0)=\varphi(x_0)=0$ and
$$\varphi_+(x)\ge u(x)\quad\text{for every $x$ in a neighborhood of $x_0$}.$$
\end{itemize}	
\end{definition}

\begin{lemma}\label{l:viscosity}
Suppose that $u$ and $v$ satisfy
$$\frac{\partial u}{\partial n}\frac{\partial v}{\partial n}=1\quad\text{on}\quad \partial\Omega\cap B_1$$
in the sense of \cref{d:solutions2}, where $\Omega=\{u>0\}=\{v>0\}$. Then, the following holds:
\begin{enumerate}[\rm(a)]
\item If $\varphi_+$ touches $\sqrt{uv}$ from below at a point $x_0\in B_1\cap\partial \Omega$, then $|\nabla \varphi(x_0)|\le 1$.
\item If $\varphi_+$ touches $\sqrt{uv}$ from above at a point $x_0\in B_1\cap\partial \Omega$, then $|\nabla \varphi(x_0)|\ge 1$.
\item If $a$ and $b$ are constants such that
$$a>0\,,\quad b>0\quad\text{and}\quad ab=1,$$
and if $\varphi_+$ touches the function $w_{ab}:=\frac12(au+bv)$ from above at $x_0\in B_1\cap\partial \Omega$, then $|\nabla \varphi(x_0)|\ge 1$.
\end{enumerate}	
\end{lemma}	
\begin{proof}
We start by proving (c). Suppose that the function $\varphi_+$ touches $w_{ab}$ from above at $x_0\in\partial\Omega$. Then, there is a ball touching $\partial\Omega$ at $x_0$ from outside (in the sense of \cref{d:one-sided-tangent-ball} (ii)). But then, by \cref{d:solutions2}, there are blow-up limits of $u$ and $v$ given respectively by
\begin{equation}\label{e:blow-ups-at-x0}
u_0(x)=\alpha \left(x\cdot\nu\right)_+\qquad\text{and}\qquad v_0(x)=\beta \left(x\cdot\nu\right)_+\ .
\end{equation}
Moreover, since $\varphi$ is smooth, the blow-up of $\varphi_+$ is given by
\begin{equation}\label{e:blow-up-varphi}
\varphi_0(x):=\big(x\cdot \nabla\varphi(x_0)\big)_+\ .
\end{equation}
By hypothesis, we have that $\varphi_0$ touches from above (at zero) the function
$$x\mapsto \frac12\Big(au_0(x)+bv_0(x)\Big)=\frac{a\alpha+b\beta}2 \left(x\cdot\nu\right)_+\ .$$
Now, since $\frac{a\alpha+b\beta}{2}\ge \sqrt{\alpha a\beta b}=1$, we have that $\varphi_0$ touches from above (again in zero) also the function
 $$x\mapsto \left(x\cdot\nu\right)_+\ .$$
 Thus, $\nabla\varphi(x_0)=\nu$ and in particular $|\nabla \varphi(x_0)|\ge 1$.

We next prove (a) and (b). If $\varphi_+$ touches $\sqrt{uv}$ from below (resp. above) at $x_0\in\partial\Omega$, then $\partial\Omega$ has an interior (resp. exterior) tangent ball at $x_0$. In particular, by \cref{d:solutions2}, $u$ and $v$ have blow-ups $u_0$ and $v_0$ given by \eqref{e:blow-ups-at-x0}. But then the function $$\sqrt{u_0v_0}= \left(x\cdot\nu\right)_+$$
is a blow-up limit of $\sqrt{uv}$. Now, using again that the blow-up limit of $\varphi_+$ is given by \eqref{e:blow-up-varphi}, we get the claim.
\end{proof}		
\begin{remark}
We will say that $u_0$ and $v_0$ are blow-up limits of $u$ and $v$ at $x_0$. We notice that the blow-up limits at a point may not be unique as they a priori depend on the sequence $r_n\to0$.
\end{remark}

\section{Statement of the main theorem}\label{s:main-main}
We now give the statement of our main theorem, which is a generalization of \cref{t:main}.
\begin{theorem}\label{t:main-main}	
Given $f,g \in L^\infty(B_1)$ non-negative and $Q \in C^{0,\alpha}(B_1)$, consider $u, v \in C(B_1)$ non-negative function which have the same support in $B_1$ and set
$\Omega:=\{u>0\}=\{v>0\}.$\\
Suppose, moreover, that $u$ and $v$ are solutions of the system
\begin{equation}\label{e:problem}
\begin{cases}
-\Delta u=f&\quad\text{in}\quad \Omega,\\
-\Delta v=g&\quad\text{in}\quad\Omega,\smallskip\\
\displaystyle{\frac{\partial u}{\partial n}\frac{\partial v}{\partial n}}=Q&\quad\text{on}\quad \partial\Omega\cap B_1,
\end{cases}
\end{equation}
where the free boundary condition holds in the sense of \cref{d:solutions2}.\\
Then, there is $\eps>0$ such that if $u$ and $v$ are $\varepsilon$-flat in $B_1$ and
$$\|f\|_{L^\infty(B_1)}+\|g\|_{L^\infty(B_1)}\le \eps^2\qquad \mbox{and}\qquad \|Q(x)-1\|_{L^\infty(B_1)}\leq \eps,$$
then $\partial\Omega$ is $C^{1,\alpha}$ in $B_{\sfrac12}$.
\end{theorem}	

\begin{remark}
We notice that the presence of $f$, $g$ and $Q$ entails only some minor technical adjustments of the proof. The main difference with respect to \cref{t:main} is in the fact that the boundary condition is given by \cref{d:solutions2} instead of \cref{d:solutions11}. In fact, what we will use in the proof is not even \cref{d:solutions2}, but the comparison properties listed in \cref{l:viscosity}. 	
\end{remark}	

\begin{remark}
The positivity assumption on $f$ and $g$ is technical and is only required for the estimate on the Laplacian of $\sqrt{uv}$ (see \cref{r:superharmonicity} below). Without this assumption one should know that the functions $u$ and $v$ are comparable on $\Omega$, i.e. that $\sfrac{u}{v}$ is bounded away from zero and infinity.
\end{remark}	

\section{A partial Harnack inequality}\label{s:partial_harnack}

In this Section we prove a Harnack type inequality for solutions to {\eqref{e:problem}} in the spirit of \cite{desilva}. In our case the boundary condition does not allow us to work separately with the solutions $u$ and $v$.  Our strategy consists in tracking the improvement of the auxiliary functions $\frac12(u+v)$ and $\sqrt{uv}$, in order to trap the boundary $\partial\Omega$ between nearby translations of a half-space. Only at this point, by exploiting the gained space, we are able to improve the estimates on the solution $(u,v)$.

\begin{remark}\label{r:superharmonicity}
Our approach relies on the fact that the two auxiliary functions are respectively subsolution and supersolution for the scalar one phase problem \eqref{e:one-phase}. Indeed, if $u$ and $v$ are harmonic in $\Omega$ and satisfy \cref{l:viscosity} (c), then  $w:=\frac12(u+v)$ is a subsolution since
$$\Delta w=0\quad\text{in}\quad\Omega,\qquad |\nabla w|\ge 1\quad\text{on}\quad \partial\Omega.$$
On the other hand, the function $z:=\sqrt{uv}$ is a supersolution:
$$\Delta z\le0\quad\text{in}\quad\Omega,\qquad |\nabla z|= 1\quad\text{on}\quad \partial\Omega.$$
The boundary condition follows again from \cref{l:viscosity}, while the superharmonicity in $\Omega$ follows from the fact that
if $u,v:\Omega\to\R$ are two positive and superharmonic functions on an open set $\Omega$, then
\begin{align*}
\Delta\big(\sqrt{uv}\big)&=\text{\rm div}\left(\frac{u\nabla v+v\nabla u}{2\sqrt{uv}}\right)
=\frac{\big(u\Delta  v+v\Delta u\big)}{2\sqrt{uv}}-\frac{\big|u\nabla v-v\nabla u\big|^2}{4(uv)^{\sfrac32}}\le \frac{\big(u\Delta  v+v\Delta u\big)}{2\sqrt{uv}}\le 0.
\end{align*}	
\end{remark}

We first show that if one is able trap the set $\Omega$ between two nearby translations of a half-space, then one can also improve the flatness estimates on $u$ and $v$. This general principle is formulated in the following lemma.
\begin{lemma}\label{t:gain_space}
Let $\eps>0$ and $\phi \in C(B_1)$ be a non-negative solution of
$$
-\Delta \phi = f\quad\mbox{in }B_1\cap \{\phi>0\},
$$
with $f\in L^\infty(B_1)$. Assume that
\begin{equation}\label{prova}
\gamma(x_d+a)_+ \le \phi(x) \le \gamma(x_d+a+\eps)_+ \quad\text{for all}\quad  x \in B_1,
\end{equation}
with $|a|\leq 1$ and either
\begin{equation}\label{e:inclusions}
\Omega\supset  \Big\{x_d+a+C\eps >0\Big\}\cap B_{\sfrac14}\qquad\text{or}\qquad \Omega\subset  \Big\{x_d+a+(1-C)\eps >0\Big\}\cap B_{\sfrac14},
\end{equation}
for some universal $C\in(0,1)$. Then, there exists $\delta\in (0,1), \rho \in (0,1)$ dimensional constants, such that either
$$
\phi(x)\geq \gamma(x_d+a+\delta\eps)_+ \qquad\mbox{or}\qquad \phi(x)\le \gamma(x_d+a+(1-\delta)\eps)_+,
$$
for every $x \in B_{\rho}$.
\end{lemma}
\begin{proof}
Let us proceed by dividing the proof in two cases.\medskip

\noindent{\it Case 1. }$\Omega\supset  \Big\{x_d+a+C\eps >0\Big\}\cap B_{\sfrac14}$. Set $D=B_{1/4}\setminus \{x_d \leq -a-C\eps\}$ and consider the function
$$
\begin{cases}
  \Delta \varphi=0 & \mbox{in } D \\
  \varphi=0 & \mbox{in } B_{1/4}\setminus D \\
  \varphi = w & \mbox{on }\partial B_{1/4}.
\end{cases},
\quad\mbox{with }
w=\phi - \frac{1}{2}(x_d+a+C\eps)_+^2 \|f\|_{L^\infty(B_1)}.
$$
Since $\phi > 0$ in $D$ we have $-\Delta \phi \geq -\|f\|_{L^\infty(B_1)}$ in $D$ and consequently $-\Delta w \geq 0$ in $D$. Then, by applying the maximum principle in $D$, we get $\varphi \leq w \leq \phi$ in $D$. On the other hand, since $\phi \geq 0 = \varphi $ in $B_{1/4}\setminus D$, it follows that $\varphi \leq \phi$ in $B_{1/4}$. Therefore, we claim that
\begin{equation}\label{des}
\varphi \geq \gamma(x_d+a+\delta\eps)_+ \quad\mbox{for all }x \in B_{1/32},
\end{equation}
for some universal $\delta\in (0,1)$, from which the desired inequality follows immediately. By \eqref{prova} we know that
$$
\gamma (x_d + a+C \eps)_+\leq \gamma \left(x_d +a + \frac{1+C}{2}\eps\right)_+ \leq \phi +\frac{1+C}{2}\gamma\eps\quad\mbox{in }\overline{B}_{1/4},
$$
and since $\varphi =\phi$ on $\partial B_{1/4}$, we get
$$
\gamma (x_d + a+C \eps)_+ \leq \varphi +\frac{1+C}{2}\gamma\eps\quad\mbox{on }\partial B_{1/4}.
$$
Therefore, by applying the maximum principle in $D$, we get
$$
\gamma (x_d + a+C \eps)_+ - \varphi \leq \frac{1+C}{2}\gamma\eps\quad\mbox{in } B_{1/4}.
$$
Consider now the function
$$
\begin{cases}
  \Delta h =0 & \mbox{in } D \\
  h=0 & \mbox{in } B_{3/16}\setminus D \\
  h = \frac{1+C}{2}\gamma\eps & \mbox{on }\partial B_{3/16}.
\end{cases}
$$
Clearly, $0\leq h \leq \frac{1+C}{2}\gamma\eps$ and by the maximum principle
\begin{equation}\label{tog}
\gamma (x_d + a+C \eps)_+ - \varphi \leq h \quad\mbox{in } B_{3/16}.
\end{equation}
Now, by the boundary Harnack inequality, if we set $\bar{x}=1/8 e_d$ we get
$$
h(x) \leq C_1 \frac{h(\bar{x})}{(1/8 + a+C \eps)_+} (x_d + a+C \eps)_+ \leq C_2 \gamma\eps (x_d+a+C\eps)_+\quad\mbox{in }B_{1/8},
$$
for some universal constants $C_1, C_2 >0$.
This last inequality, together with \eqref{tog}, leads to
$$
(1-C_2 \eps)\gamma (x_d + a+C \eps)_+ \leq \varphi \quad\mbox{in } B_{1/8}.
$$
On the other hand, since there exists $\delta\in (0,1), \delta<C$ and $\rho <1 /8$ such that
$$
(x_d + a + \delta\eps)_+ \leq (1-C_2 \eps) (x_d + a+C \eps)_+\quad\mbox{for all }x \in B_{\rho},
$$
we obtain the desired claim \eqref{des}.
\medskip

\noindent{\it Case 2. }$\Omega\subset  \Big\{x_d+a+(1-C)\eps >0\Big\}\cap B_{\sfrac14}$. Set $D=B_{1/4}\setminus \{x_d \leq -a-(1-C)\eps\}$ and consider the function
$$
\begin{cases}
  \Delta \varphi=0 & \mbox{in } D \\
  \varphi=0 & \mbox{in } B_{1/4}\setminus D \\
  \varphi = w & \mbox{on }\partial B_{1/4}
\end{cases},
\quad\mbox{with }w=\phi + \frac{1}{2}(x_d+a+(1-C)\eps)_+^2\|f\|_{L^\infty(B_1)}.
$$
Notice that $w>0$ and $-\Delta w \leq 0$ in $D$. Therefore, by maximum principle, since $w = 0$ in $B_{1/4}\setminus D$ we get  that $\varphi \geq w$ in $B_{1/4}$. Therefore, since $w\geq \phi$ in $B_{1/4}$, we have $\varphi \geq \phi$ in $B_{1/4}$. Let us claim that
$$
\varphi \leq \gamma(x_d+a+(1-\delta)\eps)_+ \quad\mbox{for all }x \in B_{1/32},
$$
for some universal $\delta\in (0,1)$, from which the desired inequality follows immediately. By \eqref{prova} we know that
$$
\phi -\frac{1+C}{2}\gamma\eps\leq \gamma \left(x_d +a + \frac{1-C}{2}\eps\right)_+ \leq  \gamma (x_d + a+(1-C) \eps)_+\quad\mbox{in }\overline{B}_{1/4},
$$
and since $\varphi =\phi$ on $\partial B_{1/4}$, we get
$$
\phi - \gamma (x_d + a+(1-C) \eps)_+ \leq \frac{1+C}{2}\gamma\eps \quad\mbox{on }\partial B_{1/4}.
$$
Therefore, by applying the maximum principle in $D$, we get
$$
\phi - \gamma (x_d + a+(1-C) \eps)_+ \leq \frac{1+C}{2}\gamma\eps \quad\mbox{in } B_{1/4}.
$$
Consider now the function
$$
\begin{cases}
  \Delta h =0 & \mbox{in } D \\
  h=0 & \mbox{in } B_{3/16}\setminus D \\
  h = \frac{1+C}{2}\gamma\eps & \mbox{on }\partial B_{3/16}.
\end{cases}
$$
Clearly, $0\leq h \leq \frac{1+C}{2}\gamma\eps$ and by the maximum principle
\begin{equation}\label{tog2}
\phi - \gamma (x_d + a+(1-C) \eps)_+ \leq  h \quad\mbox{in } B_{3/16}.
\end{equation}
Now, by the boundary Harnack inequality, if we set $\bar{x}=1/8 e_d$ we get
$$
h(x) \leq C_1 \frac{h(\bar{x})}{(1/8 + a+(1-C) \eps)_+} (x_d + a+(1-C) \eps)_+ \leq C_2 \gamma\eps (x_d+a+(1-C)\eps)_+\quad\mbox{in }B_{1/8},
$$
for some universal constants $C_1, C_2 >0$.
This last inequality, together with \eqref{tog2}, leads to
$$
\phi \leq (1+C_2)\gamma (x_d + a+(1-C) \eps)_+  \quad\mbox{in } B_{1/8}.
$$
On the other hand, since there exists $\delta\in (0,1), \delta<C$ such that
$$
 (1+C_2)(x_d + a+(1-C) \eps)_+ \leq (x_d + a+(1-\delta) \eps)_+
\quad\mbox{for all }x \in B_{R},
$$
we conclude the proof.
\end{proof}
{Finally, we can prove a partial Harnack inequality for solutions to \eqref{e:problem}.} We notice that this result will be applied to the rescalings $u_{r,x_0}$ and $v_{r,x_0}$ of the solutions $u$ and $v$ at some point $x_0\in \overline\Omega$, where as usual we set $\Omega=\{u>0\}=\{v>0\}$.
\begin{lemma}[Partial Harnack]\label{t:partial_harnack1}
	Given a constant $K>0$, there exists $\eps_0,\rho>0$ such that the following holds. If $u$ and $v$ are solutions of {\eqref{e:problem}} in the sense of \cref{d:solutions2} and are such that $0\in\overline\Omega$ and
	\begin{align*}
		\alpha(x_d+a)_+ \le u(x) \le \alpha(x_d+b)_+ \quad\text{for all}\quad  x \in B_1,\\ \beta(x_d+a)_+ \le v(x) \le \beta(x_d+b)_+\quad\text{for all}\quad  x \in B_1,
	\end{align*}
	for some $\alpha$ and $\beta$ satisfying
	  $$0<\alpha\le K,\quad 0<\beta\le K\quad\text{and}\quad \alpha\beta=1,$$
and	for some $a$ and $b$ such that
  {$$|a|< \frac1{10},\quad |b|< \frac{1}{10}\quad\text{and}\quad b-a\le \eps_0,$$}
  then there are $\tilde{a},\tilde{b}$ satisfying $\tilde{b}-\tilde{a} \le (1-\delta)(b-a)$ for some $\delta > 0$ such that
	\[
	\begin{aligned}
	\alpha(x_d+\tilde{a})_+ \le u(x) \le \alpha(x_d+\tilde{b})_+ \\ \beta(x_d+\tilde{a})_+ \le v(x) \le \beta(x_d+\tilde{b})_+
	\end{aligned},
	\]
	hold for all $x \in B_{\rho}$, with $\rho<1/8$ a dimensional constant.
\end{lemma}

\begin{proof}
As in \cite{desilva}, we fix a point $\bar{x}:= e_d/5$ and we consider the function $w: \R^d \to \R$ defined as
	\[
	w(x) := \begin{cases}1 & \text{if $x \in B_{\sfrac{1}{20}}(\bar{x})$}, \\ 0 & \text{if $x \notin B_{\sfrac{3}{4}}(\bar{x})$}, \\ \bar{c}\Big(|x-\bar{x}|^{-d} -(\sfrac{3}{4})^{-d}\Big) & \text{if $x \in B_{\sfrac{3}{4}}(\bar{x}) \diff \overline{B}_{\sfrac{1}{20}}(\bar{x})$},\end{cases}
	\]
	where $\bar{c}:=20^d-(\sfrac{4}{3})^d$. Notice that the function $w$ is nonzero exactly on $B_{\sfrac{3}{4}}(\bar{x})$ and it satisfies the following properties on $B_{\sfrac{3}{4}}(\bar{x}) \diff \overline{B}_{\sfrac{1}{20}}(\bar{x})$:
	\begin{enumerate}[label=(w-\roman*)]
		\item the function is subharmonic since $\Delta w (x)= 2d \bar{c} |x-\bar{x}|^{-d-2} \ge 2d \bar{c} (\sfrac{3}{4})^{-d-2} > 0$;
		\item $\partial_{x_d} w$ is striclty positive on the half-space $\{x_d<\sfrac{1}{10}\}$.
	\end{enumerate}

\smallskip Let us proceed by dividing the proof in three steps.\\

\noindent{\bf Step 1. Invariant transformation and flatness estimates.} Now, we consider two functions $u$ and $v$ satisfying the hypotheses of the lemma. First we notice that $$\frac1{10K}\le u(\bar x)\le 10K\qquad\text{and}\qquad\frac1{10K}\le v(\bar x)\le 10K.$$
Thus, there is a constant
$$\frac1{100K}\le c\le 100K$$
such that the couple $\widetilde u=cu$, $\widetilde v=c^{-1}v$ satisfies
$$\widetilde u(\bar x)=\widetilde v(\bar x),$$
 is also a solution to \eqref{e:equation-u}-\eqref{e:equation-v}-\eqref{e:boundary_condition}  and
\begin{equation}\label{e:flatness-tilde-u-tilde-v}
\begin{cases}
\begin{array}{ll}
\widetilde\alpha(x_d+a)_+ \le \widetilde u(x) \le \widetilde\alpha(x_d+b)_+& \quad\text{for all}\quad  x \in B_1,\\
\widetilde\beta(x_d+a)_+ \le \widetilde v(x) \le \widetilde\beta(x_d+b)_+&\quad\text{for all}\quad  x \in B_1,
\end{array}
\end{cases}
\end{equation}
where
{$$\widetilde \alpha:=c\alpha,\quad\widetilde \beta:=c^{-1}\beta\qquad\text{and}\qquad\widetilde{\alpha}\widetilde{\beta}=1\,.$$}
We define the positive constants $\delta$ and $\eps$ as follows. We set
$$\eps:=b-a<\eps_0,$$
and, without loss of generality, we assume
$$\widetilde \alpha:=1+\delta\ge 1\ge\frac1{1+\delta}=\widetilde \beta.$$
Next, we notice that since $\widetilde u(\bar x)=\widetilde v(\bar x)$, by \eqref{e:flatness-tilde-u-tilde-v}, we have
\begin{align*}
\widetilde\alpha\Big(\frac15+a\Big)\le \widetilde\beta\Big(\frac15+b\Big),
\end{align*}
and thus,
$$1+\delta=\widetilde\alpha\le \frac{\sfrac15+b}{\sfrac15+a} \widetilde\beta\le \frac{\sfrac15+b}{\sfrac15+a}= 1+\frac{b-a}{\sfrac15+a}\le 1+10\eps.$$
In particular, this implies that $\delta \leq 10\eps$ and
$$1\ge \widetilde\beta=\frac1{1+\delta}\ge 1-\delta\ge 1-10\eps.$$
Finally, we obtain
\begin{equation}\label{e:estimate-mean-alpha-beta}
1\le \frac{\widetilde\alpha+\widetilde\beta}{2}= \frac12\left(1+\delta+\frac{1}{1+\delta}\right)\le 1+\delta^2\le 1+100\eps^2.
\end{equation}
This, together with \eqref{e:flatness-tilde-u-tilde-v}, implies that for all $x\in B_1$ we have
\begin{equation}\label{e:estimate-means-tilde-u-tilde-v}
\begin{cases}
\begin{array}{rcl}
(x_d+a)_+ \le &\sqrt{\widetilde u(x)\widetilde v(x)}& \le (x_d+b)_+\smallskip\\
\ds(x_d+a)_+ \le &\ds\frac12\big(\widetilde u(x)+\widetilde v(x)\big) &\le (1+100\eps^2)(x_d+b)_+,
\end{array}
\end{cases}
\end{equation}
with
$$\widetilde u(\bar x)=\widetilde v(\bar x)=\sqrt{\widetilde u(\bar x)\widetilde v(\bar x)}=\frac{\widetilde u(\bar x)+\widetilde v(\bar x)}{2}\,.$$
Now, using again \eqref{e:flatness-tilde-u-tilde-v} and choosing $\eps_0$ such that $\delta\le 10\eps\le10\eps_0\le \frac12$, we have
{$$|\widetilde u-\widetilde v|\le 2\eps\quad\text{on}\quad B_1.$$}
Moreover, since $a,b\le 1/10$,
{$$1\ge\widetilde u\ge \frac1{40}\quad\text{on}\quad B_{\sfrac{1}{20}}(\bar x).$$}
This implies that we can choose $\eps_0$ small enough such that on $B_{\sfrac{1}{20}}(\bar x)$ we have
\begin{equation}\label{e:estimate-MA-MG}
0\le \frac{\widetilde u+\widetilde v}{2}-\sqrt{\widetilde u\widetilde v}= \widetilde u\left(1+\frac12\frac{\widetilde v-\widetilde u}{\widetilde u}\right)-\widetilde u\,\sqrt{1+\frac{\widetilde v-\widetilde u}{\widetilde u}}\le C\eps^2,
\end{equation}
where $C$ is a numerical constant.\\

\noindent{\bf Step 2. Gaining space for the domain $\Omega$.} We now reason as in \cite{desilva} and \cite{desilva-tortone} by considering two cases:\medskip

\noindent{\it Case 1.} $\ds\frac{\widetilde u(\bar x)+\widetilde v(\bar x)}{2}\ge  \frac{\eps}{2}+(\bar x_d+a)_+$. Since $|a|< 1/10$, we have $B_{1/10}(\bar x)\subset \{x_d +a >0\}$ and so the function $$h(x):=\frac{\widetilde u(x)+\widetilde v(x)}{2}-(\bar x_d+a)_+$$
is non-negative and {solves a uniformly elliptic equation} in $B_{\sfrac{1}{10}}(\bar x)$ {with right-hand side bounded from above and below by $\eps^2$}. Therefore, since $h(\bar x)\ge \frac\eps2$, the Harnack inequality gives that
$$h\ge C_{\mathcal H}\eps\quad\text{in}\quad B_{\sfrac{1}{20}}(\bar x),$$
where $C_{\mathcal H}$ is a (small) positive constant depending only on the dimension $d$. Now, using \eqref{e:estimate-MA-MG} and choosing $\eps_0$ small enough (depending on the dimension), we get that
$$\sqrt{\widetilde u\widetilde v}-(x_d+a)\ge \frac12C_{\mathcal H}\eps\quad\text{in}\quad B_{\sfrac{1}{20}}(\bar x)\,.$$
Now consider the family of functions
{$$ \psi_t(x):= x_d+a+ \frac12 C_{\mathcal{H}}\eps (w(x)-1)+\frac12 C_{\mathcal{H}}\eps t , $$}
defined for $t \ge 0$ and $x \in B_1$. So far we proved that
\begin{equation}\label{eq.sc3}
\sqrt{\widetilde u(x)\widetilde v(x)}>  \left(\psi_t(x)\right)_+ \quad \text{for every}\quad x \in B_{\sfrac{1}{20}}(\bar{x}) \quad\text{and every}\quad  t  < 1 .
\end{equation}
We will show that the same inequality holds for every $x \in B_1$. Notice that the family of functions $\psi_t$ satisfies, as a consequence of (w-i) and (w-ii) respectively, the following properties:
\begin{enumerate}[label=($\psi$-\roman*)]
	\item {$\Delta \psi_t \geq C\eps > 0$} on $B_{\sfrac{3}{4}}(\bar{x}) \diff \overline{B}_{\sfrac{1}{20}}(\bar{x})$;
	\item $|\nabla \psi_t|(x) > 1$ on $\left(B_{\sfrac{3}{4}}(\bar{x}) \diff \overline{B}_{\sfrac{1}{20}}(\bar{x})\right) \cap \{x_d < \sfrac{1}{10}\}$.
\end{enumerate}
We argue by contradiction. Suppose that for some $t < 1$ there exists $y \in B_1$ such that $\psi_t$ touches $\sqrt{\widetilde u\widetilde v}$ from below at $y$. By \cref{r:superharmonicity} and   ($\psi$-i), we have that $y \notin \Om \cap \left(B_{\sfrac{3}{4}}(\bar{x}) \diff \overline{B}_{\sfrac{1}{20}}(\bar{x})\right)$. On the other hand, by  ($\psi$-ii) and \cref{l:viscosity} we have that $y \notin \partial \Om \cap \left(B_{\sfrac{3}{4}}(\bar{x}) \diff \overline{B}_{\sfrac{1}{20}}(\bar{x})\right)$, which is a contradiction. Thus,
\begin{equation}\label{e:clean-up-from-below-MG}
\sqrt{\widetilde u(x)\widetilde v(x)}>  \left(x_d+a+ \frac12 C_\HH\eps w(x)\right)_+ \quad \text{for every}\quad  x \in B_1,
\end{equation}
and in particular,
\begin{equation}\label{e:clean-up-from-below-Omega}
\Omega\supset  \big\{x\in B_1\ :\ x_d+a+ \frac12 C_\HH\eps w(x)>0\big\}.
\end{equation}

\noindent{\it Case 2.}  $\ds\frac{\widetilde u(\bar x)+\widetilde v(\bar x)}{2}\le  \frac{\eps}{2}+(\bar x_d+a)_+$, which is equivalent to
$$\ds\frac{\eps}{2}\le  \bar x_d+b- \frac{\widetilde u(\bar x)+\widetilde v(\bar x)}{2}.$$
Using the above estimate, \eqref{e:estimate-means-tilde-u-tilde-v} and the Harnack inequality in $B_{\sfrac{1}{10}}(\bar x)$ we get that
$$(1+100\eps^2)(x_d+b)-\frac{\widetilde u+\widetilde v}{2}\ge C_{\mathcal H}\eps\quad\text{in}\quad B_{\sfrac{1}{20}}(\bar x)\,.$$
Now consider the family of functions
$$ \eta_t(x):= (1+100\eps^2)(x_d+b)-  C_{\mathcal{H}}\eps (w(x)-1)-C_{\mathcal{H}}\eps t , $$
defined for $t \ge 0$ and $x \in B_1$. Then
\begin{equation}\label{eq.sc3}
\frac{\widetilde u(x)+\widetilde v(x)}2>  \left(\eta_t(x)\right)_+ \quad \text{for every}\quad x \in B_{\sfrac{1}{20}}(\bar{x}) \quad\text{and every}\quad  t  < 1.
\end{equation}
{Let us prove that the same inequality holds for every $x \in B_1$}. Notice that, for every $t>0$, we have:
\begin{enumerate}[label=($\eta$-\roman*)]
	\item {$\Delta \eta_t < -C\eps$} on $B_{\sfrac{3}{4}}(\bar{x}) \diff \overline{B}_{\sfrac{1}{20}}(\bar{x})$;
	\item $|\nabla \eta_t|(x) < 1$ on $\left(B_{\sfrac{3}{4}}(\bar{x}) \diff \overline{B}_{\sfrac{1}{20}}(\bar{x})\right) \cap \{x_d < \sfrac{1}{10}\}$.
\end{enumerate}
As in the previous case, we argue by contradiction. Suppose that, for some $t <  1$, there exists $z \in B_1$ such that $\eta_t$ touches from above $\frac12(\widetilde u+\widetilde v)$ at $z$. By ($\eta$-i), we have that $z \notin \Om \cap \left(B_{\sfrac{3}{4}}(\bar{x}) \diff \overline{B}_{\sfrac{1}{20}}(\bar{x})\right)$. On the other hand, by  ($\eta$-ii) and \cref{l:viscosity} we have that $z \notin \partial \Om \cap \left(B_{\sfrac{3}{4}}(\bar{x}) \diff \overline{B}_{\sfrac{1}{20}}(\bar{x})\right)$, which is a contradiction. Thus,
\begin{equation}\label{e:clean-up-from-above-AG}
\frac{\widetilde u(x)+\widetilde v(x)}2\le \Big((1+100\eps^2)(x_d+b)- C_\HH\eps w(x)\Big)_+ \quad \text{for every}\quad  x \in B_1,
\end{equation}
and in particular,
{$$
\Omega\subset  \big\{x\in B_1\ :\ (1+100\eps^2)(x_d+b)- C_\HH\eps w(x)>0\big\}.
$$}
Finally, choosing $\eps_0$ small enough and using that $w$ is bounded away from zero in $B_{\sfrac{1}{4}}$, we get
\begin{equation}\label{e:clean-up-from-above-Omega-final}
\Omega\subset  \Big\{x\in B_{\sfrac14}\ :\ x_d+b- \frac12C_\HH\eps w(x)>0\Big\}\,.\vspace{0.3cm}
\end{equation}

\noindent{\bf Step 3. Conclusion.} As a consequence of Step 2, we have one of the two inclusions \eqref{e:clean-up-from-below-Omega} and \eqref{e:clean-up-from-above-Omega-final}. More precisely, there is a constant $C>0$ such that either
\begin{equation}\label{e:inclusions}
\Omega\supset  \Big\{x_d+a+C\eps >0\Big\}\cap B_{\sfrac14}\qquad\text{or}\qquad \Omega\subset  \Big\{x_d+b-C\eps >0\Big\}\cap B_{\sfrac14}.
\end{equation}
Then, by applying Lemma \ref{t:gain_space} to both $u$ and $v$ (by replacing $\gamma$ respectively with $\alpha$ and $\beta$), there exists a universal constant $\delta\in (0,1)$ such that either
$$
\begin{cases}
u(x) \le \alpha(x_d+b-\delta(b-a))_+ \\ v(x) \le \beta(x_d+b-\delta(b-a))_+
\end{cases}
\quad\mbox{or}\quad
\begin{cases}
u(x) \geq \alpha(x_d+a+\delta(b-a))_+ \\ v(x) \geq \beta(x_d+a+\delta(b-a))_+
\end{cases},
$$
for every $x \in B_{\rho}$, with $\rho < 1/8$ a dimensional constant.
\end{proof}		
{The following corollary is a consequence of the result above.}
\begin{corollary}\label{c:convergence}
Let $(u_n,v_n)$ be a sequence of solutions to {\eqref{e:problem}} in the sense of \cref{d:solutions2}.
Let
$$\Omega_n:=\{u_n>0\}=\{v_n>0\}$$
and suppose that $0\in\partial\Omega_n$, for every $n\in\N$.
Suppose that there is a sequence $\eps_n\to0$ such that for every $x\in B _1$
\begin{align*}
(x_d-\eps_n)_+ \le u_n(x) \le (x_d+\eps_n)_+\qquad\text{and}\qquad (x_d-\eps_n)_+ \le v_n(x) \le (x_d+\eps_n)_+\ .
\end{align*}
Then, there are continuous functions
$$\widetilde u_\infty:\overline B_1^+\to\R\qquad\text{and}\qquad \widetilde v_\infty:\overline B_1^+\to\R,$$
{with $B_1^+ = B_1\cap \{x_d>0\}$}, such that the following holds:
\begin{enumerate}[\rm (a)]
\item The graphs over {$\overline{\Omega}_n$} of
$$\widetilde u_n=\frac{u_n-x_d}{\eps_n}\qquad\text{and}\qquad \widetilde v_n=\frac{v_n-x_d}{\eps_n},$$
Hausdorff converge respectively to the graphs of $u_\infty$ and $v_\infty$ over {$\overline B_1^+$}.
\item The graph over {$\overline{\Omega}_n$}  of
 $$\widetilde w_n=\frac{\sqrt{u_nv_n}-x_d}{\eps_n},$$
Hausdorff converges to the graph of $\frac12\big(u_\infty+v_\infty\big)$ over {$\overline B_1^+$}.
\end{enumerate}
\end{corollary}	
\begin{proof}
The proof of claim (a) follows from \cref{t:partial_harnack1} precisely as in \cite{desilva}. In order to prove (b), we first notice that for any fixed $\delta>0$, the sequences $\widetilde u_n$ and $\widetilde v_n$ converge uniformly on $B_1\cap \{x_d>\delta\}$ respectively to the functions $\widetilde u_\infty$ and $\widetilde v_\infty$.
In particular, this implies that
\begin{align*}
\frac{\sqrt{u_nv_n}-x_d}{\eps_n}&=\frac{\sqrt{(x_d+\eps_n\widetilde u_n)(x_d+\eps_n\widetilde v_n)}-x_d}{\eps_n}={\frac12\left(\widetilde u_n+\widetilde v_n\right)}+o(\eps_n),
\end{align*}
which proves the claim on every $B_1\cap \{x_d>\delta\}$. Now, in order to have the convergence of the graphs over the whole $B_1^+$, we notice that by \eqref{e:estimate-means-tilde-u-tilde-v} the oscillation of $\sqrt{u_nv_n}-x_d$ decays when passing from $B_1$ to a smaller ball $B_\rho$. Using again the argument from \cite{desilva}, we get that the graphs of
 $$\widetilde w_n=\frac{\sqrt{u_nv_n}-x_d}{\eps_n},$$
 Hausdorff converge to the graph of a H\"older continuous function
 $$\widetilde w_\infty:\overline B_1^+\to\R.$$
 Now, since $\widetilde w_\infty={\frac12\left(\widetilde u_\infty+\widetilde v_\infty\right)}$ on each set $B_1\cap \{x_d>\delta\}$, we get that $\widetilde w_\infty={\frac12\left(\widetilde u_\infty+\widetilde v_\infty\right)}$ on $\overline B_1^+$.
\end{proof}	

\section{Improvement of flatness}\label{s:improvement_of_flatness}
{In this section we prove our main Theorem, from which the $C^{1,\alpha}$ regularity of a flat free boundary follows by standard arguments (see for example \cite{desilva-tortone} for the vectorial case). In view of the invariance of \eqref{e:problem} under suitable multiplication (see Step 1 of the proof of Lemma \ref{t:partial_harnack1}) the flatness conditions of \cref{d:flatness} can be expressed with $\alpha=\beta=1$.}

\begin{theorem}\label{t:improvement_of_flatness}
Let $(u,v)$ be solutions to {\eqref{e:problem}} in the sense of \cref{d:solutions2}.
Let
$$\Omega:=\{u>0\}=\{v>0\}$$
and suppose that $0\in\partial\Omega$. There are constants $\eps_0>0$ and $C>0$ such that the following holds. If $(u,v)$ is a couple of solutions satisfying
	$$ (x_d - \eps)_+ \le u(x) \le (x_d + \eps)_+ \qquad \text{and} \qquad  (x_d  - \eps)_+ \le v(x) \le (x_d + \eps)_+ \qquad\text{in } B_1,$$
for some $\eps < \eps_0$, then there are a unit vector $\nu \in \R^d$ with $|\nu - e_d| \le C \eps$ and a radius $\rho\in(0,1)$ such that
	$$ \tilde{\alpha}(x \cdot \nu - \sfrac{\eps}2)_+ \le u_\rho(x) \le \tilde{\alpha}(x \cdot \nu + \sfrac{\eps}2)_+ \quad \text{and} \quad\tilde{\beta}(x \cdot \nu - \sfrac{\eps}2)_+ \le v_\rho(x) \le \tilde{\beta}(x \cdot \nu + \sfrac{\eps}2)_+  $$
	for all $x \in B_1$, where $\tilde{\alpha}$ and $\tilde{\beta}$ are positive constants that satisfy
	$$\widetilde\alpha\widetilde\beta=1\ ,\qquad |1 - \tilde{\alpha}|\le C\eps\qquad\text{and}\qquad |1-\tilde{\beta}| \le C\eps.$$
\end{theorem}
We postpone the construction of the limiting problem arising from the linearization to \cref{l:linearized} and \cref{l:second-order-estimates} and we directly prove the improvement of flatness result.
\begin{proof}[Proof of Theorem \ref{t:improvement_of_flatness}]
We argue by contradiction. Let $(u_n,v_n)$ be a sequence of solutions such that
$$ (x_d - \eps_n)_+ \le u_n(x) \le (x_d + \eps_n)_+ \quad \text{and} \quad  (x_d  - \eps_n)_+ \le v_n(x) \le (x_d + \eps_n)_+ ,$$
where $\eps_n$ is an infinitesimal sequence. Let $\Omega_n:=\{u_n>0\}=\{v_n>0\}$ and consider
\begin{equation}\label{e:wildetilde-u-n-v-n}
\widetilde u_n=\frac{u_n-x_d}{\eps_n}\qquad\text{and}\qquad \widetilde v_n=\frac{v_n-x_d}{\eps_n}{\quad\mbox{on }\overline{\Omega}_n.}
\end{equation}
By the compactness result of \cref{c:convergence}, we get that $(\widetilde u_n,\widetilde v_n)$ converge, up to a subsequence, to a couple of continuous functions
\begin{equation}\label{e:wildetilde-u-v-infty}
\widetilde u_\infty: B_1\cap \{x_d\ge 0\}\to\R\qquad\text{and}\qquad \widetilde v_\infty:B_1\cap \{x_d\ge 0\}\to\R.
\end{equation}
{By \cref{l:linearized}, we have that
$$
M:=\frac12\left({\widetilde u_\infty+\widetilde v_\infty}\right)\qquad\text{and}\qquad D:=\frac12\left({\widetilde u_\infty-\widetilde v_\infty}\right)
$$
are classic solutions to
$$
\begin{cases}
\Delta M=0\quad\text{in}\quad B_1\cap\{x_d>0\},\\
\partial_{x_d} M=0\quad\text{on}\quad B_1\cap\{x_d=0\}.\end{cases}\qquad\text{and}\qquad \begin{cases}
\Delta D=0\quad\text{in}\quad B_1\cap\{x_d>0\},\\
D=0\quad\text{on}\quad B_1\cap\{x_d=0\}.\end{cases}
$$
Therefore, by the regularity result \cref{l:second-order-estimates}, we get}
$$\big|\widetilde u_\infty(x)-x\cdot \nabla \widetilde u_\infty(0)\big|\le {C_d}\rho^2\qquad\text{and}\qquad
\big|\widetilde v_\infty(x)-x\cdot \nabla \widetilde v_\infty(0)\big|\le {C_d}\rho^2\ ,$$
for every $x\in B_\rho\cap \{x_d\ge 0\}$. Now, we can write this as
$$
\begin{cases}x\cdot \nabla \widetilde u_\infty(0)-{C_d}\rho^2\le \widetilde u_\infty(x)\le x\cdot \nabla \widetilde u_\infty(0)+ {C_d}\rho^2\medskip\\
x\cdot \nabla \widetilde v_\infty(0)-{C_d}\rho^2\le \widetilde v_\infty(x)\le x\cdot \nabla \widetilde v_\infty(0)+ {C_d}\rho^2
\end{cases}\qquad\text{for every}\qquad x\in B_\rho\cap \{x_d\ge 0\}.$$
Now, this implies that for $n$ large enough
$$
\begin{cases}x\cdot \nabla \widetilde u_\infty(0)-2{C_d}\rho^2\le \widetilde u_n(x)\le x\cdot \nabla \widetilde u_\infty(0)+ 2{C_d}\rho^2\medskip\\
x\cdot \nabla \widetilde v_\infty(0)-2{C_d}\rho^2\le \widetilde v_n(x)\le x\cdot \nabla \widetilde v_\infty(0)+2{C_d}\rho^2
\end{cases}\qquad\text{for every}\qquad x\in B_\rho\cap \overline\Omega_n,$$
which by the definition of $\widetilde u_n$ and $\widetilde v_n$ can be written as
$$x\cdot \Big(e_d+\eps_n\nabla \widetilde u_\infty(0)\Big)-\eps_n2{C_d}\rho\le  (u_n)_\rho(x)\le x\cdot \Big(e_d+\eps_n\nabla \widetilde u_\infty(0)\Big)+\eps_n2{C_d}\rho\ ,$$
$$x\cdot \Big(e_d+\eps_n\nabla \widetilde v_\infty(0)\Big)-\eps_n2{C_d}\rho\le  (v_n)_\rho(x)\le x\cdot \Big(e_d+\eps_n\nabla \widetilde v_\infty(0)\Big)+\eps_n2{C_d}\rho\ ,$$
for every $x\in B_1\cap \big(\frac1\rho\Omega_n\big)$. We next set
$$V:=\nabla M(0)\qquad\text{and}\qquad c:=\frac{\partial D}{\partial x_d}(0).$$
Thus,
$$x\cdot \Big(e_d+\eps_nV+c\eps_n e_d\Big)-\eps_n2{C_d}\rho\le  (u_n)_\rho(x)\le x\cdot \Big(e_d+\eps_nV+c\eps_n e_d\Big)+\eps_n2{C_d}\rho\ ,$$
$$x\cdot \Big(e_d+\eps_nV-c\eps_n e_d\Big)-\eps_n2{C_d}\rho\le  (v_n)_\rho(x)\le x\cdot \Big(e_d+\eps_nV-c\eps_n e_d\Big)+\eps_n2{C_d}\rho\ .$$
Now, since {by Lemma \ref{l:linearized}} $V$ and $e_d$ are orthogonal, we can compute
$$|e_d(1\pm c\eps_n)+\eps_nV|=\sqrt{1\pm 2c\eps_n+\eps_n^2(c^2+|V|^2)}=1\pm c\eps_n+O(\eps_n^2).$$
Then, fixing $\rho>0$ small enough and taking $\eps_n$ sufficiently small with respect to $\rho$, we get
$$x\cdot \frac{e_d+\eps_nV}{|e_d+\eps_nV|}-\frac12\eps_n\le  \frac{1}{1+c\eps_n}(u_n)_\rho(x)\le x\cdot \frac{e_d+\eps_nV}{|e_d+\eps_nV|}+\frac12\eps_n\  ,$$
$$x\cdot \frac{e_d+\eps_nV}{|e_d+\eps_nV|}-\frac12\eps_n\le  {(1+c\eps_n)}(v_n)_\rho(x)\le x\cdot \frac{e_d+\eps_nV}{|e_d+\eps_nV|}+\frac12\eps_n\ ,$$
for every $x\in B_1\cap \big(\frac1\rho\Omega_n\big)$. { Finally, the contradiction follows by taking
$$
\nu =\frac{e_d+\eps_nV}{|e_d+\eps_nV|}, \quad \tilde{\alpha}=1+c \eps_n\quad\mbox{and}\quad \tilde{\beta} =\tilde{\alpha}^{-1}.
$$}
\end{proof}	
{Under the notations of the proof of Theorem \ref{t:improvement_of_flatness}, we introduce the limiting problem arising from the linearization near flat free boundary points.}
\begin{lemma}[The linearized problem]\label{l:linearized}
Let $\ \widetilde u_\infty$ and $\widetilde v_\infty$ be as in \eqref{e:wildetilde-u-v-infty} and set
{$$M:=\frac12\left({\widetilde u_\infty+\widetilde v_\infty}\right)\qquad\text{and}\qquad D:=\frac12\left({\widetilde u_\infty-\widetilde v_\infty}\right).$$}
Then, $M$ and $D$ are classical solutions of
\begin{equation}\label{e:lineariz}
\begin{cases}
\Delta M=0\quad\text{in}\quad B_1\cap\{x_d>0\},\\
\partial_{x_d} M=0\quad\text{on}\quad B_1\cap\{x_d=0\}.\end{cases}\qquad\text{and}\qquad \begin{cases}
\Delta D=0\quad\text{in}\quad B_1\cap\{x_d>0\},\\
D=0\quad\text{on}\quad B_1\cap\{x_d=0\}.\end{cases}
\end{equation}
\end{lemma}
\begin{proof}	
We divide the proof in several steps.\medskip

\noindent{\it Step 1. Equations in $\{x_d>0\}$.} First we notice that the equation
$$\Delta M=\Delta D=0\quad\text{in}\quad B_1\cap\{x_d>0\}$$
follows from the fact that on every compact subset of $B_1\cap\{x_d>0\}$, the functions $\widetilde u_n$ and $\widetilde v_n$ given by \eqref{e:wildetilde-u-n-v-n}, are harmonic and
converge uniformly to $\widetilde u_\infty$ and $\widetilde v_\infty$. \medskip

\noindent{\it Step 2. Boundary condition for $D$.} In order to prove the boundary condition
$$D=0\quad\text{on}\quad B_1\cap\{x_d=0\},$$
we notice that the graphs of $\widetilde u_n$ and $\widetilde v_n$ over $\partial\Omega_n$ are both given by the graph of the function $-\frac{1}{\eps_n}x_d$. Thus, by the Hausdorff convergence of the graphs, we get that
$$u_\infty=v_\infty\quad\text{on}\quad B_1\cap\{x_d=0\}.$$
\smallskip

Finally, it remains to prove that
$$\frac{\partial M}{\partial x_d}=0\quad\text{on}\quad B_1\cap\{x_d=0\}$$
is satisfied in the  viscosity sense.\\ Notice that, the fact that $M$ is a classical solution of \eqref{e:lineariz} follows by \cite[Lemma 2.6]{desilva}).\\\smallskip

\noindent{\it Step 3. The boundary condition for $M$ from below.} Suppose that a quadratic polynomial $P$ touches $M$ strictly from below at a point $x_0 \in \{x_d=0\}$. We will show that $\partial_d P(x_0)\le0$. Therefore, suppose by contradiction that
\begin{equation}\label{e:assurdo-from-below}
\frac{\partial P}{\partial x_d}(x_0)>0,
\end{equation}
and notice that we can assume that $\Delta P>0$ in a neighborhood of $x_0$. Let now
$$\widetilde w_n:=\frac{\sqrt{u_nv_n}-x_d}{\eps_n}\,:\,\overline\Omega_n\to\R.$$
By \cref{c:convergence}, we have that the sequence of graphs of $\widetilde w_n$ over $\overline\Omega_n$ converge in the Hausdorff sense to the graph of $M$ over $B_1\cap\{x_d\ge 0\}$.
In particular, this means that the graph of $P$ touches from {below} the graph of $\widetilde w_n$ at some point $x_n\in \overline\Omega_n$. Since $\widetilde w_n$ is superharmonic in $\Omega_n$ (see \cref{r:superharmonicity}), we have that necessary $x_n\in\partial\Omega_n$. But then, we have
$$P(x)\le \frac{\sqrt{u_nv_n}-x_d}{\eps_n}\quad\text{for}\quad x\in\overline\Omega_n,$$
with an equality when $x=x_n$, which can be written as
$$\eps_nP(x)+x_d\le \sqrt{u_n(x)v_n(x)}\quad\text{for}\quad x\in\overline\Omega_n.$$
Setting
\begin{equation}\label{e-linearized-equation-varphi}
\varphi(x):=\eps_nP(x)+x_d,
\end{equation}
we get that $\varphi_+$ touches
$\sqrt{u_nv_n}\ $
from below at $x_n$. On the other hand, by \cref{l:viscosity} we know that
$$1+\eps\frac{\partial P}{\partial x_d}(x_n)=\partial_{x_d}\varphi(x_n)\le |\nabla \varphi(x_n)|\le 1,$$
which is a contradiction with \eqref{e:assurdo-from-below}, as we claimed.\medskip

\noindent{\it Step 4. The boundary condition for $M$ from above.} Suppose that a quadratic polynomial $P$ touches $M$ strictly from above at a point $y_0\in \{x_d=0\}$. We will prove that $\partial_d P(y_0)\ge0$. Conversely, suppose that
\begin{equation}\label{e:assurdo-from-above}
\frac{\partial P}{\partial x_d}(y_0)<0,
\end{equation}
and notice that we can assume that $\Delta P<0$ in a neighborhood of $y_0$. By the Hausdorff convergence of the graphs, the graph of $P$ touches from above the graph of {$\frac12(\widetilde u_n+\widetilde v_n)$} at a point $y_n \in \overline\Omega_n$. Since $\widetilde u_n$ and $\widetilde v_n$ are harmonic in $\Omega_n$ we have that $y_n\in\partial\Omega_n$. But then, we have
{$$P(x)\ge \frac12\left(\frac{u_n-x_d}{\eps_n}+\frac{v_n-x_d}{\eps_n}\right)\quad\text{for}\quad x\in\overline\Omega_n,$$}
with an equality when $x=y_n$, which can be written as
$$x_d + \eps_n P(x)\ge \frac{u_n(x)+v_n(x)}{2}\quad\text{for}\quad x\in\overline\Omega_n.$$
Let now
\begin{equation}\label{e-linearized-equation-varphi}
\psi(x):=x_d + \eps_n P(x),
\end{equation}
Then, $\psi_+$ touches $\frac12\big(u_n(x)+v_n(x)\big)$ from above at $y_n$. Thus, by \cref{l:viscosity},
$$1\le |\nabla \psi(y_n)|^2 
=1+2\eps_n\frac{\partial P}{\partial x_d}(y_n)+\eps_n^2|\nabla P(y_n)|^2,$$
which can be written as
$$\frac{\partial P}{\partial x_d}(y_n)+\frac{\eps_n}2|\nabla P(y_n)|^2\ge 0,$$
Passing to the limit as $n\to\infty$, we get $\displaystyle\frac{\partial P}{\partial x_d}(y_0)\ge 0,$ which is a contradiction.
\end{proof}	
{\begin{remark}
For the vectorial Bernoulli problem, in \cite{krli, desilva-tortone} the authors proved that the linearized problem arising by the improvement of flatness technique is a system of decoupled equations in which the first component satisfies a Neumann problem, while the others have Dirichlet boundary conditions. On the contrary, in our case, the nonlinear formulation of the free boundary condition \eqref{e:boundary_condition} requires to consider suitable linear combinations of the solutions $u, v$ in order to detect the problem solved by the limits $u_\infty, v_\infty$.
\end{remark}
}
{For the sake of completeness, we briefly sketch the proof of the decay for the solutions of the linearized problem \eqref{e:lineariz}, which we used in the proof of \cref{t:improvement_of_flatness}.}
\begin{lemma}[First and second order estimates for $\widetilde u_\infty$ and $\widetilde v_\infty$]\label{l:second-order-estimates}
Let $\widetilde u_\infty$ and $\widetilde v_\infty$ be as in \cref{l:linearized}. Then $\widetilde u_\infty$ and $\widetilde v_\infty$	are $C^\infty$ in $B_1\cap\{x_d\ge 0\}$ and we have the estimates
	\begin{equation}\label{e:gradient_estimate}
\|\nabla \widetilde u_\infty\|_{L^\infty(B_{\sfrac{1}2}\cap\{x_d\ge 0\})}+\|\nabla \widetilde v_\infty\|_{L^\infty(B_{\sfrac{1}2}\cap\{x_d\ge 0\})}\le {C_d}\ ,
\end{equation}
and
\begin{equation}\label{e:second_derivative_estimate}
\begin{cases}
\big|\widetilde u_\infty(x)-x\cdot \nabla \widetilde u_\infty(0)\big|\le {C_d}|x|^2\smallskip\\
\big|\widetilde v_\infty(x)-x\cdot \nabla \widetilde v_\infty(0)\big|\le {C_d}|x|^2
\end{cases}
\qquad\text{for every}\qquad x\in B_{\sfrac{1}{2}}\cap\{x_d\ge 0\}\ .
\end{equation}
Moreover, we have
\begin{equation}\label{e:gradient-torsion}
\nabla\widetilde u_\infty(0)=\nabla\widetilde v_\infty(0)+e_d\frac{\partial (\widetilde u_\infty-\widetilde v_\infty)}{\partial x_d}(0).
\end{equation}
\end{lemma}	
\begin{proof}	
Let $M$ and $D$ be as in \cref{l:linearized}. Then, both can be extended to (resp. an even and an odd) harmonic functions in the ball $B_1$. We can the use the classical gradient and second order estimates (see for instance \cite[Lemma 7.17]{velectures}) for a harmonic function $h:B_R\to\R$, that is,
	\begin{equation*}
	\|\nabla h\|_{L^\infty(B_{\sfrac{R}2})}\le \frac{C_d}{R}\|h\|_{L^\infty{(B_R)}}\ ,
	\end{equation*}
	and
	\begin{equation*}
	\big|h(x)-h(0)-x\cdot \nabla h(0)\big|\le \frac{C_d}{R^2}|x|^2\|h\|_{L^\infty{(B_R)}}\qquad\text{for every}\qquad x\in B_{\sfrac{R}{2}}\ ,
	\end{equation*}
	where $C_d$ is a dimensional constant. Now since $\widetilde u_\infty(0)=\widetilde v_\infty(0)=0$ and
	$$|\widetilde u_\infty|\le 1\qquad\text{and}\qquad |\widetilde v_\infty|\le 1\qquad\text{in}\qquad B_1\cap\{x_d\ge 0\},$$
and	since $M=\widetilde u_\infty+\widetilde v_\infty$ and $D=\widetilde u_\infty-\widetilde v_\infty$, we get that
	\begin{equation*}
\begin{cases}
	\big|D(x)-x\cdot \nabla D(0)\big|\le {C_d}|x|^2\smallskip\\
	
	\big|M(x)-x\cdot \nabla M(0)\big|\le {C_d}|x|^2
	\end{cases}\qquad\text{for every}\qquad x\in B_{\sfrac{1}{2}}\cap\{x_d\ge 0\}\ ,
	\end{equation*}
	which gives \eqref{e:second_derivative_estimate}. Finally, \eqref{e:gradient-torsion} follows from the fact that $D\equiv 0$ on $\{x_d=0\}$.
\end{proof}

\end{document}